\documentclass[12pt]{amsart}

\usepackage{amscd,amssymb,amsfonts,amsmath,amstext,amsthm}

\newcommand{\ot}{\otimes}

\newcommand{\lan}{\langle}
\newcommand{\ran}{\rangle}

\newcommand{\eps}{\varepsilon}

\newcommand{\s}{\sigma}
\newcommand{\La}{\Lambda}
\newcommand{\vs}{\vskip 3pt}
\newcommand{\ds}{\displaystyle}
\newcommand{\om}{\omega}
\newcommand{\la}{\lambda}

\def\CS{{\mathcal S}}

\def\CO{{\mathcal O}}
\def\CB{{\mathcal B}}
\def\CL{{\mathcal L}}
\def\CM{{\mathcal M}}

\def\fp{{\mathfrak p}}

\def\Z{{\mathbb Z}}
\def\C{{\mathbb C}}

\def\Q{{\mathbb Q}}
\def\E{{\mathbb E}}
\def\L{{\mathbb L}}
\def\K{{\mathbb K}}

\renewcommand\k{\Bbbk}

\newtheorem{thm}{Theorem}[section]
\newtheorem{df}[thm]{Definition}
\newtheorem{remark}[thm]{Remark}
\newtheorem{ex}[thm]{Example}
\newtheorem{cor}[thm]{Corollary}
\newtheorem{prop}[thm]{Proposition}
\newtheorem{lem}[thm]{Lemma}
\newtheorem{ques}[thm]{Question}

\newcommand{\num}{\refstepcounter{thm}}

\title{Modular Representations and Indicators for Bismash Products}

\author{Andrea Jedwab}
\address{University of Southern California, Los Angeles, CA 90089-1113}
\email{jedwab@usc.edu}
\author{Susan Montgomery}
\address{University of Southern California, Los Angeles, CA 90089-1113}
\email{smontgom@math.usc.edu}

\begin{document}

\thanks{The first author was supported by NSF grant DMS 0701291 and the second 
author by DMS 1001547.}

\begin{abstract} We introduce Brauer characters for representations of the bismash products
 of groups in characteristic $p > 0,$ $p \neq 2$ and study their properties analogous to the classical 
 case of finite groups. We then use our results to extend to bismash products a theorem of Thompson 
 on lifting Frobenius-Schur indicators from characteristic $p$ to characteristic 0.
\end{abstract}
 
\maketitle


\section{Introduction}

In this paper we study the representations of bismash products $H_{\k} = \k^G\# \k F$, coming 
from a factorizible group of the form $ Q= FG$ over an algebraically closed field $\k$ of 
characteristic $p > 0$, $p \neq 2$. Our general approach is to reduce the problem to a corresponding 
Hopf algebra in characteristic 0. 

In the first part of the paper, we extend many of the classical facts about Brauer characters of groups in char $p>0$ to the case of our bismash products; our Brauer characters are defined on a special subset 
of $H$ of non-nilpotent elements, using 
the classical Brauer characters of certain stabilizer subgroups $F_x$ of the group $F$. In particular 
we relate the decomposition matrix of a character for the bismash product in char 0 with respect to 
our new Brauer characters, to the ordinary decomposition matrices for the group algebras of the 
$F_x$ with respect to their Brauer characters. As a consequence we are able to extend a 
theorem of Brauer saying that the determinant of the Cartan matrix for the above decomposition 
is a power of $p$ (Theorem \ref{Hbrauer}). 

These results about Brauer characters may be useful for other work on modular representations. 
We remark that the only other work on lifting from characteristic $p$ to characteristic $0$ of which 
we are aware is that of \cite{EG}, and they work only in the semisimple case.

In the second part, we first extend known facts on Witt kernels for $G$-invariant forms to the case of a 
Hopf algebra $H$, as well as some facts about $G$-lattices. We then use these results and Brauer 
characters to extend a theorem of J. Thompson \cite{Th} on Frobenius-Schur indicators for representations of finite groups to the case of bismash product Hopf algebras. In particular we show 
that if $H_{\C} = \C^G\# \C F$ is a bismash product over $\C$ and $H_{\k} = \k^G\# \k F$ is the corresponding bismash product over an algebraically closed field $\k$ of characteristic $p>0$, and 
if $H_{\C}$ is totally orthogonal (that is, all Frobenius-Schur indicators are +1), then the same is true 
for $H_{\k}$ (Corollary \ref{orth}). 

This paper is organized as follows. Section 2 reviews known facts about bismash products and their representations, and Section 3 summarizes some basic facts about Brauer characters for representations of finite groups. In Section 4 we prove our main results about Brauer characters for the case of bismash products. 

In Section 5 we extend the facts we will need on Witt kernels and lattices, and in Section 6 we combine 
all these results to prove our extension of Thompson's theorem. Finally in Section 7 we give some 
applications and raise some questions. 

Throughout $\E$ will be an arbitrary field and $H$ will be a finite dimensional Hopf algebra
over $\E$, with
comultiplication $\Delta:H \to H \otimes H$ given by $\Delta(h)=\sum
h_1\otimes h_2$, counit $\epsilon:H\to \E$ and antipode $S:H\to H$.


\section{Extensions arising from factorizable groups and their representations}

The Hopf algebras we consider here were first described by G. Kac \cite{Ka} in the setting of 
$C^*$-algebras, in which case $\E = \C$, and in general by Takeuchi \cite{Ta}, constructed 
from what he called a matched pair of groups. These Hopf algebras can also be constructed from 
a factorizable group, and that is the approach we use here. Throughout, we assume that $F$ and 
$G$ are finite groups.

\begin{df}
{\rm A group $Q$ is called \textit{factorizable} into subgroups $F,G \subset Q$ if $FG= Q$ and 
$F\cap G = 1$; equivalently, every element $q \in Q$ may be written uniquely as a product $q = ax$ 
with $a\in F$ and $x\in G$}. 
\end{df}

A factorizable group gives rise to actions of each subgroup on the other. That is, we have 
$$ \rhd: G \times F \to F \text{ and } \lhd: G \times F \to G,$$ 
where for all $x \in G, \ a \in F$,  the images $x\rhd a\in F$ and $x\lhd a\in G$ are the (necessarily 
unique) elements of $F$ and $G$ such that $xa = (x\rhd a)(x\lhd a)$. 

Although these actions $\rhd$ and $\lhd$ of $F$ and $G$ on each other are not group automorphisms, 
they induce actions of $F$ and $G$ as automorphisms of the dual algebras $\E^G$ and $\E^F$. 
Let $\{p_x\; |\; x\in G\}$ be the basis of $\E^G$ dual to the basis $G$ of $\E G$ and let 
$\{p_a\; |\; a\in F\}$ be the basis of $\E^F$ dual to the basis $F$ of $\E F$. Then the induced actions are 
given by 
\begin{equation}\num \label{action}
a\cdot p_x := p_{x\lhd a^{-1}} \text{  and  } \ x\cdot p_a := p_{x\rhd a}, 
\end{equation}
for all $a\in F$, $x\in G$. We let $F_x$ denote the stabilizer in $F$ of $x $ under the action $\lhd$.

The {\it bismash product} Hopf algebra $H_{\E}:= \E^G\# \E F$ associated to $Q = FG$ uses the actions 
above. As a vector space, 
$H_{\E} = \E^G\otimes \E F$, 
with $\E$-basis $\{p_x\# a \; | \; x\in G, a\in F\}$. The algebra structure is the usual smash product, given by 
\begin{equation}\num \label{mult}
(p_x\# a)(p_y\# b)= p_x (a\cdot p_y)\# ab = p_x p_{y\lhd a^{-1}} \# ab = \delta _{y,{x\lhd a}}p_x\# ab.
\end{equation}
The coalgebra structure may be obtained by dualizing the algebra structure of $H_{\E}^*$, although we will only need here that $H_{\E}$ has counit $\epsilon (p_x\# a)=\delta _{x,1}$. Finally the antipode of $H$ 
is given by $S(p_x\# a)= p_{(x\lhd a)^{-1}}\# (x\rhd a)^{-1}$. One may 
check that $S^2=id$.

For other facts about bismash products, including the alternate approach 
of matched pairs of groups, see \cite{Ma2}, \cite{Ma3}. We will consider the explicit example of a 
factorization of the symmetric group in Section 7.

Observe that for any field $\E$, a {\it distinguished basis} of $H_{\E}$ over $\E$ is the set 
\begin{equation} \num\label{basis}
\CB := \{ p_y \# a \ | \   y \in G, \ a \in F \},
\end{equation}
and that $\CB$ has the property that if $b, b' \in \CB$, then $b b' \in \CB \cup \{0\}.$ In particular, if 
$w =  p_y \# a $, then (\ref{mult}) implies that for all $k \geq 2$, 
\begin{equation} \num\label{power}
w^k =  \left\{
          \begin{array}{ll}
          p_y \# a^k &\text{if } a \in F_y\\
          0	           &\text{if } a \notin F_y.
          \end{array}
          \right.
\end{equation} 
Thus if $a \in F_y$ and has order $m$, the minimum polynomial of $p_y \# a$ is 
$$f(Z) = Z^{m+1} - Z,$$ and so the characteristic roots of $p_y \# a$ are $\{0\} \cup \{m^{th}\text{ roots of 1} \}$.

\begin{lem} \label{SB}(1) $\CB$ is closed under the antipode $S$. 

(2) The set $\CB' := \{ p_y \# a \in \CB \ | \  a \in F_y \}$ 
is also closed under $S$.

(3) If $w = p_y \# a \in \CB'$, then $S(w) = p_{y^{-1}} \# y a^{-1}y^{-1}$. 
\end{lem}

\begin{proof} (1) is clear from the formula for $S$ above. For (2), formula (\ref{power}) shows that 
$\CB'$ is exactly the set of non-nilpotent elements of $\CB$, so it is also closed under $S$. 

For (3), $w \in \CB'$ implies that $a \in F_y$, and thus $y\lhd a = y$. Then 
$$ya = (y\rhd a)(y\lhd a) = (y\rhd a) y$$ 
and so $y\rhd a = y a y^{-1}$. Substituting in the formula for $S$, we see 
$S(w) = p_{y^{-1}} \# y a^{-1}y^{-1}.$
\end{proof}

We review the description of the simple modules over a bismash product. 

\begin{prop} \label{mod} Let $H = \E^G \# \E F$ be a bismash product, as above, where now $\E$ is algebraically closed. 
For the action $\lhd $ of $F$ on $G$, fix one element $x$ in each $F$-orbit $\CO$ of $G$, and let 
$F_x$ be its stabilizer in $F$, as above. Let $V = V_x$ be a simple left $F_x$-module and let 
$\hat{V}_x = \E F\ot_{\E F_x} V_x$ denote the induced $\E F$-module. 

$\hat{V}_x$ becomes an $H$-module in the following way: for any $y\in G$, $a,b\in F$, and $v\in V_x$, 
$$(p_y\#a)[b\ot v]=\delta_{y\lhd(ab), x} (ab\ot v).$$
Then $\hat{V}_x$ is a simple $H$-module under this action, and every simple $H$-module arises 
in this way. 
\end{prop}

\begin{proof} In the case of characteristic 0, this was first proved for the Drinfel'd double $D(G)$ 
of a finite group $G$ over $\C$ by \cite{DPR} and \cite{M}. The case of characteristic $p>0$ was done by \cite{W}. 

For bismash products, extending the results for $D(G)$, the characteristic 0 case was done in 
 \cite[Lemma 2.2 and Theorem 3.3]{KMM}. The case of characteristic $p>0$ follows by extending the arguments of \cite{W} for $D(G)$; see also \cite{MoW}.
\end{proof} 

\begin{remark} {\rm The arguments  for Proposition \ref{mod} also show that if we begin with an 
indecomposable module $V_x$ of $F_x$, then $\hat{V}_x$ is an indecomposable module for $H_{\E}$, 
and all indecomposable $H_{\E}$-modules arise in this way. This fact is discussed in \cite{W2} after Proposition 4.4; it could also be obtained from \cite{W1}, using the methods of Theorem 2.2 and 
Corollary 2.3 in that paper. }
\end{remark}

Now fix an irreducible $H_{\E}$-module $\hat V= \hat{V}_x = \E F\ot_{\E F_x} V_x$  as in Proposition \ref{mod}.  To compute the values of the character for $\hat{V}$, we use a 
formula from \cite{JM}; it is a simpler version of \cite[Proposition 5.5]{N2} and is similar to the 
formula in \cite[p 898]{KMM}:

\begin{lem} \cite[Lemma 4.5]{JM} \label{lemma3} Fix a set $T_x$ of representatives for the 
right cosets of $F_x$ in $F$. Let $\chi_x$ be the character of $V_x$. Then the character $\hat{\chi}_x$ of $\hat{V}_x$ may be computed as follows: 
$$\hat{\chi}_x(p_y\#a) = \sum_{t\in T_x \text{ and }\newline t^{-1}at \in F_x} \delta_{y\lhd t,x} \chi_x(t^{-1}at),$$
for any $y\in G,\ a \in F$.
\end{lem}

Next we review some known facts about Frobenius-Schur indicators for representations of Hopf algebras. For a representation $V$ of $H$, recall that a bilinear form $\lan - ,-  \ran : V \ot_{\E} V \to \E$ is 
{\it $H$-invariant} if for all $h \in H$ and $v, w \in V$,
$$\sum \lan h_1\cdot v , h_2\cdot w \ran  = \eps(h) 1_{\E}.$$ 
It follows that the antipode is the adjoint of the form; that is, for all $h, l \in H$, $v, w \in V$,
$$\lan  S(h)\cdot v , l \cdot w \ran  = \lan h\cdot v , \bar{S}(l)\cdot w  \ran = \lan h\cdot v ,S(l)\cdot w  \ran,$$
using that $S^2 = id$. 

\begin{thm} \label{GMind} \cite{GM} Let $H$ be a finite-dimensional Hopf algebra over $\E$ such that $S^2 = id$ and $\E$ splits $H$. Let $V$ be an irreducible representation of $H$. Then $V$ has a well-defined Frobenius-Schur indicator $\nu(V) \in \{ 0, 1, -1\}$. Moreover

(1) $\nu(V) \neq 0 \iff V^* \cong V$. 

(2) $\nu(V) = +1$ (respectively -1) $\iff$ $V$ admits a non-degenerate $H$-invariant 
symmetric (resp, skew-symmetric) bilinear form. 
\end{thm}

If in addition $H$ is semisimple and cosemisimple, then in fact $\nu(V)$ can be computed by 
the formula $\nu(V) = \chi(\La_1 \La_2)$, where $\chi$ is the character belonging to $V$ and 
$\La$ is a normalized integral of $H$ \cite{LM}. This formula does not work in general, but still 
Theorem  \ref{GMind} applies to any bismash product since as noted above, it is always true that
$S^2 = id$. Sometimes the indicator is called the {\it type} of $ V$. 

We remark that, unlike the case for groups, $\nu(V) = +1$ does not imply that the character $\chi_V$ is 
real-valued, even when $\E = \C$. However it is still true that if $V ^* \cong V$, then $\chi^* = \chi$, 
that is, $\chi \circ S = \chi$.

Finally we fix the following notation: 

\begin{df} \label{pmod}  \cite[p 402]{CR}. {\rm  A {\it  $p$-modular system} $(\K, R, \k)$ consists of a discrete valuation ring $R$ with quotient field $\K$, maximal ideal $\fp = \pi R$ containing the rational prime $p$, and residue class field 
$\k = R/ \mathfrak{p}$ of characteristic $p$. }
\end{df}

We will mainly be interested in the following special case, as in \cite{Th} with a slight change in 
notation. 

\begin{ex} \label{mainex} {\rm Let $H_{\Q}$ be a Hopf algebra over $\Q$ and let $\L$ be an 
algebraic number field which is a splitting field for $H_\Q$. Let $\mathcal P$ be a prime ideal of the 
ring of integers of $\L$ containing the rational prime $p\neq 2$, let $R$ be the completion of the ring of 
$\mathcal P$-integers of $\L$, $\K$ be the field of fractions of $R$, $\pi$ be a generator for the 
maximal ideal $\fp$ of $R$, and $\k = R/\pi R$. 

Then $(\K, R, \k)$ is a $p$-modular system.}
\end{ex}


\section{Brauer characters for $G$} \label{Brauer}

In this section we review the definition of Brauer characters for a finite group  \cite{CR}, \cite{Nv} 
and summarize some of the classical results. 

We fix the following notation, for a given finite group $G$. 
Let $|G|$ denote the order of $G$, and let $|G|_p$ denote the largest power of $p$ in $|G|$; 
thus $|G| = |G|_p m$ where $p \nmid m$.  Since $\K$ splits $G$, it contains a primitive $m^{th}$ 
root of 1, say $\om$, which in fact is in $R$. Under the natural map $f: R \to \k$, 
$\bar \om = f(\om)$ is a primitive $m^{th}$ root of 1 in $\k$.

Let $G_{p'}$ denote the set of {\it $p$-regular elements} of $G$, that is elements of $G$ whose order is 
prime to $p$. Thus for each $x \in G_{p'}$, all of the eigenvalues of $x$ on any (left) 
$\k G$-module $W$ are $m^{th}$ roots of 1, and hence may be expressed as a power of 
$\bar \om$. Denote the eigenvalues of $x$ by $\{ \bar \om^{i_1}, \ldots,  \bar \om^{i_t} \}.$

\begin{df} \label{Bchar} {\rm For each (left) $\k G$-module $W$, the $\K$-valued function 
$\phi: G_{p'} \to \K$ defined for 
each $x \in G_{p'}$ by
$$\phi(x) =  \om^{i_1}  + \cdots + \om^{i_t} = \sum_{j=1}^t  f^{-1} ( \bar{\om}^{i_j}).$$
is called the {\it Brauer character of $G$ afforded by $W$}.}
\end{df}

\begin{remark} \label{lincom} {\rm Note that $\phi$ is a class function on the conjugacy classes of $p$-regular elements of $G$. $\phi$ can be extended to $\phi^{\#}$, a class function on all of $G$, by defining $\phi^{\#}(x) =0$ 
for any $x$ in the complement of $G_{p'}$. It follows that $\phi^{\#}$ is a $\K$-linear combination of the ordinary 
irreducible characters $\chi_i$ of $\K G$ \cite[p 423]{CR}. Thus $\phi$ is a $\K$-linear combination of the $\chi_i |_{G_{p'}}.$} 
\end{remark}

We record some facts about Brauer characters of groups. See \cite[17.5]{CR}, \cite[Chapter 15]{I}.

\begin{prop}\label{CR17.5}
\begin{enumerate}
 \item
$\la$ takes values in $R$ and $\overline{\la(x)}=Tr(x, V),$ all $ x\in G_{p'}.$

\item
Let $W_0\supset W_1\supset 0$ be $\k G$-modules, let $\phi$ be the Brauer character afforded by $\ds{W_0 / W_1},\  \phi_1$ the Brauer character afforded by $W_1 $ and $\phi_0$ the Brauer character of 
$W_0$. Then $\phi_0=\phi + \phi_1.$

\item
Let $V$ be a $\K G$-module with $\K$-character $\chi$. Then for each $RG$-lattice $M$ in $V$, 
the restriction $\chi |_{G_{p'}}$ is the Brauer character of the $kG$-module $\overline{ M} := M/ \fp M$.

\end{enumerate}
\end{prop}

We fix the following notation, as in \cite{CR}:

\noindent (1) $Irr(G) =\{\chi_1, \ldots, \chi_n\}$ denotes the irreducible characters of $\K G$;

\noindent (2) $Irr_{\k}(G) = \{\psi_1, \ldots, \psi_d\}$ denotes the irreducible characters of $\k G$;

\noindent (3) $IBr(G) = \{\phi_1, \ldots, \phi_d\}$ denotes the Brauer characters corresponding to 
$\{\psi_1, \ldots, \psi_d\}$.  

\vs

By \cite[16.7 and 16.20]{CR}, there exists a homomorphism of abelian groups
\begin{equation} \num \label{demap}  d: G_0(\K G) \to G_0(\k G),
\end{equation}
called the {\it decomposition map}, such that for any class $[V]$ in $G_0(\K G)$, $d([V]) = [\overline{M}]  \in 
 G_0(\k G),$ where $M$ is any $RG$-lattice in $V$ and $\overline{ M} := M/ \fp M$. 
 
Using Proposition \ref{CR17.5}(3), it follows that for any $\chi_i$, there are integers $d_{ij}$ such that 
\begin{equation} \num \label{decomp} \chi_i |_{G_{p'}} = \sum_j d_{ij} \phi_j.
\end{equation}
The mulitplicities $d_{ij} = d(\chi_i|_{G_{p'}} , \phi_j)$ are called  {\it decomposition numbers}, and the matrix 
$D = [ d_{ij}]$ is called the {\it decomposition matrix}.  The matrix $C = D^t D$ is called the {\it Cartan matrix}.

From \cite{CR}, 17.12 - 17.15, the set $Bch(\k G)$ of {\it virtual Brauer characters}, 
that is $\Z$-linear combinations of Brauer characters of $\k G$-modules, is a ring under addition and multiplication of functions, and $Bch(\k G) \cong G_0(\k G).$ Using that 
$G_0(\K G) \cong ch(\K G)$, the ring of virtual characters of $\K G$, it follows that the decomposition map $d$ induces a map 
$$d': ch(\K G) \to Bch(\k G), $$ 
where $d'$ is the restriction map $\psi \to \psi |_{G_{p'}} $.

Consequently Equation (\ref{decomp}) implies that if $\chi$ is the character for $V$ and 
$ \chi |_{G_{p'}} = \sum_j \alpha_j \phi_j$, where $\phi_j$ is the Brauer character of the simple 
$\k G$-module $W_j$,  and $d([V]) = [\overline{M}]  \in G_0(\k G),$ then 
\begin{equation} \label{decompmodp} [\overline{M}]  =  \sum_j \alpha_j [W_j].
\end{equation}

We will need the following theorem in our more general situation:

\begin{thm} (Brauer) \label{brauer} \cite[18.25]{CR}\cite[Ex (15.3)]{I} $Det(C)$ is a power of $p$. 
\end{thm} 

We will also need the analog of the following:

\begin{thm} \cite[(17.9)]{CR} \label{indep}The irreducible Brauer characters $IBr(G)$ form a $\K$-basis of the space of $\K$-valued class functions of $G_{p'}.$
\end{thm}

One consequence of this theorem is:

\begin{cor} Let $\E$ be a splitting field for $G$ of char $p>0$.  Then the number of 
simple $\E G$-modules is equal to the number of $p$-regular conjugacy classes of $G$.
\end{cor}

A crucial ingredient of the proof of the theorem  is the following elementary lemma.

\begin{lem} \label{factor} Let $\rho: G \to GL_n(\k)$ be a matrix representation of $G$ over $\k$. 
For any $x \in G$, we may write $x = us$, where $u$ is a $p$-element of $G$ and $s$ is a 
$p'$-element. Then $x$ and $s$ have the same eigenvalues, counting multiplicities. \end{lem} 
 
The lemma follows since $su = us$, and all eigenvalues of $u$ will equal 1.



\section{Brauer characters for $H_{\k}$} \label{Hbrauer}

In this section we define Brauer characters for our bismash products and show that they have properties 
analogous to those for finite groups discussed in Section \ref{Brauer}. 

Assume that $\L, \K, \pi, R$ and $\k$ are as Example \ref{mainex}, with $\k= R/ \pi R$.

Fix an irreducible $H_{\L}$-module $V_{\L}$ whose indicator is non-zero. 
Since $\L$ is a splitting field for $H_{\Q}$, so is $\K$, and thus 
$$V := V_{\L} \ot _{\L} \K$$ 
is an irreducible $H_\K$-module. Moreover the bilinear form on $V_{\L}$ extends to a bilinear form on 
$V$, and 
thus there is a non-singular $H_{\K}$-invariant bilinear form 
$\lan \ ,\  \ran$  on $V$, with values in $\K$, which is symmetric or skew-symmetric by \ref{GMind}.	

Recall the basis $\CB$ of $H_{\K}$ from Section 2. 

\begin{df} \label{lat}{\rm Let $V = V_{\L} \ot _{\L} \K$  be as above. Then an 
$R \CB$-{\it lattice} in $V$ is a finitely generated $R \CB$-submodule $L$ of $V$ such that 
$\K L = V$.}
\end{df} 

From now on we also assume that $\L$  denotes an algebraic number field which is a splitting field for $H_{\Q} = \Q^G\# \Q F$. Then $\k$ is a splitting field for $H_{\k}$.

Let $\hat{W}=\hat{W}_x$ be a simple $H_{\k}$ module, as in Proposition \ref{mod}. That is, for a given $F$-orbit 
$\CO$ of $G$ and fixed $ x\in \CO = \CO_x$, with $F_x$ the stabilizer of $x$ in $F$ and $W = W_x$ a simple $\k F_x$-module, 
$\hat{W} = \k F \ot_{\k F_x} W.$
Recall $\hat{W}$ becomes an $H$-module via
$$(p_y\#a)[b\ot w]=\delta_{y\lhd(ab), x} [ab\ot w],$$
 for any $y\in G$, $a,b\in F$, and $w\in W$. 
 
 As in Lemma \ref{lemma3}, fix a set $T_x$ of representatives of the right cosets of 
$F_x$ in $F$.

\begin{lem} \label{stab} Consider the action of $p_y\#a$ on $\hat W = \hat{W}_x$ as above. 

(1) If $(p_y\#a)\hat{W} \neq 0$, then there exists  $t \in T_x$ and $w \in W$ such that 
$$(p_y\#a)[t\ot w]=\delta_{y\lhd(at), x} [at\ot w] \neq 0.$$
Thus $y  = x \lhd(at)^{-1} \in \CO_x.$ 

(2) If $p_y\#a$ has non-zero eigenvalues on $\hat{W}_x$, then $a \in F_y$ and $x = y \lhd t,$ where 
$t$ is as in (1). 

(3) For $t$ as in (1) and (2), $t^{-1} a t \in F_x$ and so $at \ot w = t \ot (t^{-1} a t)w.$

\end{lem} 

\begin{proof}

(1) By the formula for the action of $p_y \#a$ on $\hat W$, there exists $b \in F$ and $w \in W$ such that $(p_y\#a)[b\ot w]=\delta_{y\lhd(ab), x} [ab\ot w] \neq 0.$ Thus $y  = x \lhd(ab)^{-1} \in \CO_x.$ Now for 
some $t \in T_x$, $b \in tF_x$. It is easy to see that $t$ satisfies the same properties as $b$. 

(2)  If $p_y\#a$ has non-zero eigenvalues on $\hat{W}$, then $(p_y\#a)^2 \neq 0$, 
and so $a \in F_y$ by (\ref{power}). Now using (1), 
$x  = y \lhd(at) = (y \lhd a)\lhd t = y\lhd t.$ 

(3) Since $ y = x \lhd t^{-1}$ and $a \in F_y$, it follows that  $t^{-1} a t \in F_x$. Thus we can 
write $at\otimes w = t \otimes (t^{-1}at)w$.
\end{proof}

We next prove an analog of Lemma \ref{factor}, although in our case the two factors do not 
necessarily commute in $H_{\k}.$

\begin{lem} \label{Hfactor}
Consider $\rho: H_{\k} \to End_k(\hat W) \cong M_n(\k)$. For $a\in F$ write $a=su,$ with $s$ the $p$-regular part and $u$ the $p$-part of $a.$  Then $\rho(p_y\#a)$ and $\rho(p_y\#s)$ have the same eigenvalues, counting multiplicities. 
\end{lem}

\begin{proof} 
Even though $p_y\#s$ and $1\#u$ do not commute, their actions on $\hat{W}$ do commute:
suppose  $b\otimes w$ is such that $(p_y \# a)\cdot [b\otimes w] \neq 0.$ By Lemma \ref{stab}, 
$y \lhd b = x$ and $a \in F_y$. Thus  $s \in F_y$ since $s$ is a power of $a$. Then
\begin{eqnarray*}{}
(p_y \# s)(1\#u)\cdot [b\otimes w] & = & (p_y \# a)\cdot [b\otimes w]\\
          					 & = & ab\otimes v
 \end{eqnarray*}
 
and 
\begin{eqnarray*}
(1\#u)(p_y \# s)\cdot [b\otimes w] & = & \delta_{y\lhd sb, x}(1\# u)\cdot [sb\otimes w]\\
          					 & = & (1\#u)\cdot [sb\otimes w]\\
					 	 & = & usb\otimes w \\
						 & = & ab\otimes w.
           \end{eqnarray*}         
Since the eigenvalues of $1 \#u$ are all $1,$ the eigenvalues of $p_y\#a$ are the same as 
those of $p_y\#s.$
\end{proof}

The lemma shows that to find the character of some $p_y\#a$, it suffices to look at the character of 
$p_y\#s$, where $s$ is the $p'$-part of $a$. Moreover, by Lemma \ref{stab}, the character of $p_y\#a$ 
will be non-zero only if $a \in F_y$. 

Thus, as a replacement for the $p'$-elements of the group in the classical case, we consider the subset of the basis $\CB$ defined in (\ref{basis}) of those elements which are non-nilpotent element and have 
group element in $F_{p'}$. That is, we define 
\begin{equation} \num\label{p'basis}
\CB_{p'} := \{ p_y \# a \in \CB' \ | \  a \in F_{p'} \} = \{ p_y \# a \in \CB \ | \  a \in F_y \cap F_{p'} \},
\end{equation} 
where $F_{p'}$ is the set of $p$-regular elements in $F$. By Lemma \ref{SB}, $\CB_{p'}$ is also closed 
under the antipode $S$, since if $a \in F_{p'}$ and $w = p_y \# a$ is non-nilpotent, then by Lemma 
\ref{SB}(3), $S(w) = p_{y^{-1}} \# y a^{-1}y^{-1}$. Since $y a^{-1}y^{-1}$ has the same order as $a$, $S(w)$ is also in $\CB_{p'}$.

The above remarks motivate our definition of Brauer characters for $H_{\k},$ by using the formula 
 in Lemma \ref{lemma3}. That is, if $W = W_x$ is a simple $\k F_x$-module with character $\psi$, then 
 the character of the simple $H_{\k}$-module $\hat{W}$ is given by 
 \begin{equation}\num \label{hatpsi}
 \hat{\psi}(p_y\#a)=\sum_{t \in T_x \text{ and }t^{-1}at\in F_x}\delta_{y\lhd t, x} \ \psi(t^{-1}at).
 \end{equation}

\begin{df}\label{Hchar} {\rm  Let $W = W_x$ be a simple $\k F_x$-module with character $\psi$, and let $\phi$ be the classical Brauer character of $W$ constructed from $\psi$. Then the {\it Brauer character} of $\hat{W}$ is the function $$\hat{\phi}: \CB_{p'} \to \K$$
defined on any $p_y\#a \in \CB_{p'}$ by}
$$\hat{\phi}(p_y\#a)=\sum_{t \in T_x \text{ and }t^{-1}at\in F_x}\delta_{y\lhd t, x} \ \phi(t^{-1}at).$$
\end{df}

\begin{remark} {\rm  If $\hat{\phi}$ is a Brauer character, then also $\hat{\phi}^* = \hat{\phi} \circ S$ 
is a Brauer character: namely if $\hat{\phi}$ is the Brauer character for $\hat{\psi}$, then 
$\hat{\phi}^*$ is the Brauer character of $\hat{\psi}^*$, using the fact that $\CB_{p'}$ 
is stable under $S$.}
\end{remark}

\vs

We fix the following notation, as for groups:

\vs

\noindent (1)  $Irr(H_{\K}) =\{\hat{\chi_1}, \ldots, \hat{\chi_n}\}$ denotes the irreducible characters of $H_{\K}$; 

\noindent (2) $Irr_{\k}(H_{\k}) = \{\hat{\psi_1}, \ldots, \hat{\psi_d}\}$ denotes the irreducible 
characters of $H_{\k}$; 

\noindent (3) $IBr(H_{\k}) = \{\hat{\phi_1}, \ldots, \hat{\phi_d}\}$ denotes the Brauer characters corresponding to $\{\hat{\psi_1}, \ldots, \hat{\psi_d}\}$. As for groups, the elements of $IBr(H_{\k})$ 
are called {\it irreducible} Brauer characters. 

\noindent  (4) $Bch(H_{\k})$  denotes the ring of {\it virtual} Brauer characters of $H_{\k}$, that is, 
the $\Z$-linear span of the irreducible Brauer characters. 

\vs

\begin{lem} \label{Hlincom} $\hat{\phi}_j$ is a $\K$-linear combination of the $\hat{\chi}_i  |_{\CB_{p'}}.$
Consequently if all $\hat{\chi}_i$ are self-dual, then all $\hat{\phi}_j$ are also self-dual, and so are all 
$\hat{\psi}_j$.
\end{lem}

\begin{proof} By Remark \ref{lincom} applied to $F_x$, the Brauer character $\phi_j$ may be written as 
$\phi_j = \sum_i    \alpha_i  \ \chi_i |_{(F_x)_{p'}}, $ for $\alpha_i \in \K$. Lifting this equation through induction up to $F_{p'}$ (and so to  $\CB_{p'}$) as in Lemma \ref{lemma3},  we obtain the first 
statement in the lemma. 

Now if all $\hat{\chi}_i$ are self-dual, then the same property holds for the $\hat{\phi}_j$ since they 
are linear  combinations of the $\hat{\chi}_i  |_{\CB_{p'}}.$ Fix one of the $\hat {\psi}_j$ and its 
Brauer character $\hat{\phi_j}$. Since $\hat {\phi}_j^* = \hat {\phi}_j\circ S$ and 
$\hat {\psi}_j^* = \hat {\psi}_j\circ S$, using the formula for $S$ as well as (\ref{hatpsi}) and the 
formula in \ref{Hchar}, we see that $\hat {\phi}_j^* = \hat {\phi}_j$ if and only if 
$\hat {\psi}_j^* = \hat {\psi}_j$.  
\end{proof} 

We may follow exactly the proof of Proposition \ref{CR17.5}, that is \cite[17.5, (2) - (4)]{CR}, to show the following:

\begin{prop} \label{CR17.5b}(1) $\hat{\phi}$ takes values in $R$ and $\overline{\hat{\phi}(p_y\#a))}=Tr(p_y\#a,\hat{W})$, for $a\in F_{p'}.$

(2) Given $H_{\k}$-modules $\hat{W_0}\supset \hat{W_1}\supset 0$, let $\hat{\phi}$ be the Brauer character afforded by $\ds{\hat{W_0} \ / \ \hat{W_1}},\ \hat{\phi_1}$ the Brauer character afforded by $\hat{W_1},$ and $\hat{\phi_0}$ the Brauer character of $\hat{W_0}$. Then 
$\hat{\phi_0}=\hat{\phi}+\hat{\phi_1}.$

(3) Let $V$ be a $\K \CB$-module with $\K$-character $\chi$. Then for each $R\CB$-lattice $M$ in $V$, 
the restriction $\chi |_{R\CB_{p'}}$ is the Brauer character of the $H_{\k}$-module $\overline{ M} := M/ \fp M$.
\end{prop}

Similarly, one may follow the first part of the proof of Theorem \ref{indep} \cite[17.9]{CR}, replacing 
Lemma \ref{factor} with Lemma \ref{Hfactor}, to show

\begin{thm} \label{Hindep} The irreducible Brauer characters $IBr(H_{\k})$ are $\K$-linearly independent. \end{thm}

We may also extend the decomposition map $d$ in Section 3 to obtain a homomorphism of abelian groups
\begin{equation} \num \label{demap}  \hat{d}: G_0(H_{\K}) \to G_0(H_{\k} ),
\end{equation}
called the {\it decomposition map}, such that for any class $[\hat{V}]$ in $G_0(H_{\K} )$, 
$\hat{d}([\hat{V}]) = [\overline{M}]  \in 
 G_0(H_{\k} ),$ where $M$ is any $R\CB$-lattice in $\hat{V}$ and $\overline{ M} := M/ \fp M$.

Again using the facts about groups, the set $Bch(H_{\k} )$ of {\it virtual Brauer characters}, 
that is $\Z$-linear combinations of Brauer characters of $\k G$-modules, is a ring under addition and multiplication of functions, and $Bch(H_{\k} ) \cong G_0(H_{\k} ).$ Using that 
$G_0(\K G) \cong ch(\K G)$, the ring of virtual characters of $H_{\K}$, it follows that the decomposition map $\hat{d}$ induces a map 
$$\hat{d'}: ch(H_{\K }) \to Bch(H_{\k }), $$ 
where $\hat{d'}$ is the restriction map $\hat{\psi} \to \hat{\psi} |_{\CB_{p'}} $.

Consequently Equation (\ref{decomp}) implies that if $\hat{\chi}$ is the character for $\hat{V}$ and 
$\hat{ \chi} |_{\CB_{p'}} = \sum_j \alpha_j \hat{\phi}_j$, where $\hat{\phi}_j$ is the Brauer character of the simple 
$H_{\k }$-module $\hat{W}_j$,  and $\hat{d}([\hat{V}]) = [\overline{M}]  \in G_0(H_{\k }),$ then 
\begin{equation} \label{decompmodp} [\overline{M}]  =  \sum_j \alpha_j [W_j].
\end{equation}

From now on we wish to distinguish the characters (over $\k$ or $\K$) which arise from stabilizers 
of elements in different $F$-orbits of $G$. Assume that there are exactly $s$ distinct orbits of $F$ on 
$G$ and that we fix $x_q \in \CO_q$, the $q^{th}$ orbit. Thus for a fixed $x= x_q \in G$ with stabilizer $F_x = F_{x_q}$, we will write 
$\chi_{i, x}$ for an irreducible character of $\K F_x$, and $\hat{\chi}_{i, x}$ for its induction up to $\K F$, 
which becomes an irreducible character of $H_{\K}$. 

Similarly $\psi_{j, x}$ denotes an irreducible character of $\k F_x$, and $\hat{\psi}_{j, x}$ its induction up to $\k F$, 
which becomes an irreducible character of $H_{\k}$. Also $\phi_{j,x}$ denotes the Brauer 
character corresponding to $\psi_{j, x}$, and $\hat{\phi}_{j,x}$ the Brauer character corresponding to 
$\hat{\psi}_{j, x}$.

\begin{lem} Let $\phi_x$ be a virtual Brauer character of $\k F_x$ and assume that 
$\phi_x = \sum_j z_{j,x} \phi_{j,x}$, where as above the $\phi_{j,x}$ are the Brauer characters of $\k F_x$. 

Then $\hat{\phi}_x = \sum_j z_{j,x} \hat{\phi}_{j,x}.$
\end{lem}

The lemma follows from Definition \ref{Hchar} of a Brauer character $\hat{\phi}$ for $H_{\k}$ in terms of  a Brauer character $\phi$ for $\k F_x$. 
Moreover Lemma \ref{lemma3} becomes 
$$\hat{\chi}_{i,x}(p_y\#a) = \sum_{t\in T_x \text{ and }\newline t^{-1}at \in F_x} \delta_{y\lhd t,x} \chi_{i, x} (t^{-1}at) .$$

 Applying Equation (\ref{decomp}) to $F_x,$ there are integers $d_{{ij},x}$ such that 
$$ \chi_{i,x} |_{(F_x)_{p'}} = \sum_j d_{{ij}, x} \phi_{j,x} ,$$ 
where the $ \phi_{j,x} $ are in $IBr(\k F_x)$.

Lifting the $\chi_{i,x}$ to $\hat \chi_{i,x}$ on $\CB_{p'}$, we see that 
$$\hat{ \chi}_{i ,x} |_{\CB_{p'}} = \sum_j d_{{ij},x} \hat{\phi}_{j,x}.$$ 

That is, the decomposition numbers for the $\hat{ \chi}_{i ,x} |_{\CB_{p'}} $ with resepct to the 
$\hat{\phi}_{j,x}$ are the same as the decomposition numbers for the 
$ \chi_{i ,x} |_{(F_x)_{p'}}$ with respect to the $\phi_{j,x}$ 
for the group $F_x$. Thus the decomposition matrix $\hat{D}_x = [ d_{{ij},x}]$ for  the 
$\hat{ \chi}_{i ,x} |_{\CB_{p'}}$ with respect to the $\hat{\phi}_{j,x}$ is the same as the decomposition matrix $D_x$ for the $ \chi_{i ,x} |_{(F_x)_{p'}}$ with respect to the $\phi_{j,x}.$ 

The above discussion proves 

\begin{prop} As above, assume that there are exactly $s$ distinct orbits $\CO$ of $F$ on $G$ and 
choose $x_q \in \CO_q$, for $q = 1, \dots,s$. Then 

\noindent (1) $\hat{D}_{x_q} = D_{x_q}$

\noindent (2) The decomposition matrix for the $\hat{ \chi}_i |_{\CB_{p'}}$ with respect to the 
$\hat{\phi}_j$ is the block matrix 
 $\hat{D} = \left [\begin{array}{llll}
\hat{D}_{x_1} & 0 & \cdots & 0\\
0 & \hat{D}_{x_2} & \cdots & 0\\
0  & 0 & \cdots & 0\\
0 &0 &\cdots & \hat{D}_{x_s} \end{array}\right]$
where $\hat{D}_{x_q}$ is the decomposition matrix of $\hat{ \chi}_{i ,x_q} |_{\CB_{p'}}$ with respect to 
$\hat{\phi}_{j, x_q}$. 
\end{prop} 

As for groups, $\hat{C} = \hat{D}^t \hat{D}$ is called the {\it Cartan matrix}. We are now able to extend 
the theorem of Brauer we need (\ref{brauer}). 

\begin{thm} \label{Hbrauer} $Det(\hat{C})$ is a power of $p$. 
\end{thm} 

\begin{proof} First, $\hat C$ is also a block matrix, with blocks 
$\hat{C}_{x_q} = (\hat{D}_{x_q})^t \hat{D}{x_q}$. By Brauer's theorem applied to each group
 $F_{x_q}$, we know that $Det(\hat{C}_{x_q})$ is a power of $p$. Thus $Det (\hat C)$ is a power of $p$. 
\end{proof} 

\vs


\section{Invariant Forms: Witt kernels and Lattices} \label{witt}

A first step in the direction of extending Thompson's theorem concerns the Witt kernel of a module 
with a bilinear form as in Theorem \ref{GMind}. 
We will show that, for an arbitrary field $\E$, the notion of Witt kernel of an $\E G$-module extends 
to $H_{\E}$-modules. One can then follow the argument in \cite{Th}. 

Let $V$ be a finitely-generated $H_{\E}$-module which is equipped with a non-degenerate 
$H_{\E}$-invariant bilinear form 
$\lan - ,-  \ran : V \ot_{\E} V \to \E$, which is either symmetric or skew symmetric. For example, if $\E$ is algebraically closed, then any irreducible self-dual $H_{\E}$-module has such a form by Theorem \ref{GMind}. For any submodule $U$ of $V$, 
$$U^{\perp} = \{ v \in V \ | \ \lan v, U\ran = 0\}.$$ 
Since the form is $H$-invariant and $S$ is the adjoint of the form, for all $u \in U^{\perp},$
$$\lan h \cdot v, U\ran = \lan v, S(h)\cdot U\ran = \lan v, U\ran =0.$$ 
Thus $U^{\perp}$ is also a submodule of $V$. Note also that $U^{\perp \perp} = U$ since $V$ is 
finite-dimensional over $\E$. Let
$$\CM = \CM_V = \{V_0 \ | \ V_0 \text{ is an $H_{\E}$-submodule of $V$ and }\lan V_0 , V_0\ran = 0\}, $$
that is, $V_0 \subseteq V_0^{\perp}$.
 
Obviously, $\{0\} \in \CM$, and $\CM$ is partially ordered by inclusion. If $V_0 \in \CM$, then 
$V_0^{\perp}/ V_0$ inherits a non-degenerate form given by
$$(v_0+V_0,v'_0+V_0)_{V_0^{\perp}/ V_0} := \lan v_0, v'_0\ran, \ \ v_0, v'_0 \in V_0.$$
If $V_1$ is a maximal element of $\CM$, it is not difficult to see that $V_1^{\perp}/V_1$ is a completely 
reducible $ H_E$-module,  and the restriction of $  (\ , \ )_{V_1^{\perp}/ V_1} $ to any $ H_E$-submodule of 
$V_1^{\perp}/V_1$  is non-degenerate.

\begin{df} {\rm Let $V_1$ be a maximal element of $\CM$. Then 
the {\it Witt kernel} of $V$ is $V' :=  V_1^{\perp}/V_1$.} \end{df}

It is not clear from this definition that the Witt kernel is independent of the choice of the maximal element of 
$\CM$. However, we have

\begin{lem}\label{Witt}  If $V_1$, $V_2$ are maximal elements of $\CM$, then there is an $H_{\E }$-isomorphism
$$\Phi: V_1^{\perp}/V_1 \to V_2^{\perp}/V_2,$$
such that
$$ (v_1, v'_1)_{V_1^{\perp}/V_1} = 	(\Phi (v_1 ), \Phi(v'_1) )_{V_2^{\perp}/V_2} , \text{ for all } v_1, v'_1 \in V_1^{\perp}/V_1.$$
\end{lem}

The proof follows exactly the proof of \cite[Lemma 2.1]{Th} for group algebras.

\vs

We next  extend the facts shown in \cite{Th} about $G$-invariant forms on $RG$-lattices to the case of lattices for bismash products. Our proofs follow \cite{Th} very closely. 

Assume that $\L, \K, \pi, R$ and $\k$ are as at the end of Section 2, with $\k= R/ \pi R$.

Fix an irreducible $H_{\L}$-module $V_{\L}$ whose indicator is non-zero. 
Since $\L$ is a splitting field for $H_{\Q}$, so is $\K$, and thus 
$$V := V_{\L} \ot _{\L} \K$$ 
is an irreducible $H_\K$-module. Moreover the bilinear form on $V_{\L}$ extends to a bilinear form on 
$V$, and 
thus there is a non-singular $H_{\K}$-invariant bilinear form 
$\lan \ ,\  \ran$  on $V$, with values in $\K$, which is symmetric or skew-symmetric by \ref{GMind}.	

Recall the basis $\CB$ of $H_{\K}$ from Section 2. 

\begin{df} \label{lat}{\rm Let $V = V_{\L} \ot _{\L} \K$  be as above. Then an 
$R \CB$-{\it lattice} in $V$ is a finitely generated $R \CB$-submodule $L$ of $V$ such that 
$\K L = V$.}
\end{df} 

Let $\CL = \CL_V$ be the family of $R \CB$-sublattices of $V$. If $L \in \CL$, then $L^*$ denotes 
the {\it dual lattice} defined by
$$L^* = \{ l \in V \ | \ \lan L, l \ran \subseteq R \}.$$

Since $R$ is Noetherian, $L^*$ is also an $R \CB$-lattice by \cite[4.24]{CR}. In particular $L^*$ 
is also finitely-generated. We also let 
$$\CL_I  = \CL_{V, I}= \{ R \CB\text{-lattices } L \in \CL_V  \ | \  \lan L, L \ran \subseteq R \}$$
denote the set of integral lattices. 
If $L$ is any element of $\CL$, there is an integer $n$ such that $\pi^n L \in \CL_I$. Obviously, 
$\CL_I$ is partially ordered by inclusion and if $L_1, L_2 \in \CL_I$ with $L_1 \subseteq  L_2$, 
then $L_2 \subseteq L_1^*.$ Thus any chain of sublattices starting with $L_1$ is contained in 
$L_1^*, $ which is a Noetherian $R \CB$-module, and so the chain must stop. 
Thus every element of
$\CL_I$ is contained in a maximal element of $\CL_I$. 

In the following discussion, $L$ denotes a {\it fixed} maximal element of $\CL_I$.

\begin{lem}\label{form} (1) $\pi L^*\subseteq  L.$ 

(2) Let $M = L^*/L$. There is a non singular $R \CB$-invariant form $\lan \ , \ \ran_M$
on $M$, with values in $\k$, defined as follows: if $m_1 , m_2 \in M$, $m_i = x_i+ L$ 
then $\lan m_1,m_2\ran_M := \text{image in $\k$ of }\pi\lan x_1,x_2\ran.$
\end{lem} 

\begin{proof} (1) Let $h$ be the smallest integer $\geq 0$ such that $\pi^h L^* \subseteq L$. 
If $h\leq 1$, then (1) holds. So suppose $h\geq 2.$

Let $L_1= L+\pi^{h-1} L^*$. Then $L_1 \in \CL$. Moreover, if $u_1,u_2 \in L_1$, say
$u_i = l_i + \pi^{h-1} l_i^*$, $l_i \in L$, $l_i^* \in L^*$,  then
$$\lan u_1, u_2\ran = \lan l_1, l_2\ran + \pi^{h-1} (\lan l_1,  l_2^*\ran + \lan l_1^*, l_2\ran) +
\pi^{h-2} \lan \pi^h l_1^*, l_2^*\ran \in R$$
by definition of $L^*$ and of $h$. Thus, $L_1 \in \CL_I$. Since $L\subseteq L_1$, this violates 
the maximality of $L$. So (1) holds.

(2) If $l_1^*, l_2^* \in  L^*$, then $\pi l_1^* \in L$ , so $\lan\pi l_1^*, l_2^*\ran \in R$. Since 
$\lan L, L^*\ran$ and $\lan L^*,L\ran$ are contained in $R$, and since $\pi$ is a generator 
for the maximal ideal of $R$, it follows that $\lan \ , \ \ran_M$ is well defined. To see that 
this form is non singular, suppose 
$l^* \in L^*$ and $\lan l^*,L^*\ran = 0$. Then $l^* \in L^{**}  =L$, so $l^*+L=0$ in $M$. This proves (2).
\end{proof}

\begin{lem} \label{forY} $\{ l \in L\ | \ \lan l,L\ran \subseteq\pi R \}=\pi L^*.$
\end{lem}

\begin{proof} By Lemma \ref{form}(1) , $\pi L^* \subseteq L.$  By definition of $L^*$,  
$\pi L^* \subseteq \{ l \in L \ | \ \lan l,L\ran \subset \pi R \}.$ Thus it suffices to show that if $l \in L$ 
and $\lan l,L\ran \subseteq \pi R$, then $l \in \pi L^*.$ This is clear, since 
$\lan 1/{\pi} l, L \ran \subseteq R,$ so that by the definition of $L^*$, we have $1/{\pi} l \in  L^*.$
\end{proof}


\section{ Indicators and Brauer characters}

In this section we combine our work in the previous sections to prove the analog of a theorem of Thompson. 

\begin{thm} {\rm{Thompson} \cite{Th}}   Let $k$ be an algebraically closed field of odd characteristic, 
let $G$ be a finite group, and let $W$ be an irreducible $\k G$-module. If $W$ has non-zero 
Frobenius-Schur indicator, then $W$ is a composition factor (of odd multiplicity) in the reduction 
mod $p$ of an irreducible $\K G$-module with the same  indicator as $W.$
\end{thm}

By reduction mod $p$, we mean to use the $p$-modular system $(\K, R, \k)$ as described in 
Example \ref{mainex}, and then the induced decomposition map as in (\ref{decompmodp}).

We first prove the analog of \cite[Lemma 3.3]{Th}. Recall the notation in Section 4:  

$V_{\L}$ is a fixed irreducible $H_{\L}$-module which is self-dual and thus 
$V = V_{\L} \ot _{\L} \K$ is an irreducible self-dual $H_{\K}$-module, with character $\chi$. $V$ has a 
non-degenerate $H_{\K}$-invariant bilinear form $\lan \ ,\  \ran$ with values in $\K$, which is symmetric or skew-symmetric by Theorem \ref{GMind}.	

As in Section \ref{witt}, $\CL$ is the family of $R \CB$-sublattices of $V$ and $\CL_I $ is the 
subset of integral lattices.
Let $L$ denote a {\it fixed} maximal element of $\CL_I$ with dual lattice $L^*$.

Consider the following $H_{\k}$-modules: let $X = L^*/ \pi L^* $, 
$Y = L/\pi L^*$, and $Z = L^*/L$. Note that $Y$ is a submodule of $X$. Then there is 
an exact sequence of $H_{\k}$-modules 
$$
0 \to Y\to X\to Z \to 0.
$$

Using the non-degenerate form on $V$, it follows from the argument in Lemma \ref{form}(2) that both 
$X$ and $Z$ have a non-degenerate form; moreover Lemma \ref{forY} 
gives us a non-degenerate form on $Y$. These three forms are all of the same type, that is, either all 
are symmetric or all are skew symmetric, and the type is given by the indicator $\nu(\chi)$ of $V$.

\begin{prop} \label{old3.3} Let $V$, $L$, $X$, $Y$, and $Z$ be as above. Suppose that $P$ is an irreducible $H_{\k}$-module with Brauer 
character $\hat{\phi}$, such that

(1)  $\hat{\phi}^* = \hat\phi$; 

(2) $d(\hat{\chi} |_{\CB_{p'}}, \hat{\phi})$ is odd. 

\noindent Let $Y'$, $Z'$ be the Witt kernels of $Y$, $Z$ respectively. Then the multiplicity of $P$ in 
$Y' \oplus Z'$ is odd. Consequently $P$ has the same type as $V$. 
\end{prop}

\begin{proof}  We let $M = L^*$; then $X = L^* / \pi L^*  =  M / \fp M.$ By Proposition 
\ref{CR17.5b}, the restriction $\chi |_{R\CB_{p'}}$ is the Brauer character of the $H_{\k}$-module 
$\overline{ M} := M/ \fp M$.

By hypothesis, the multiplicity of $\hat{\phi}$ in $\hat{\chi} |_{\CB_{p'}}$ is odd, and thus using the decomposition map, the multiplicity of $P$ in $X = \overline{M}$ is odd. Thus the multiplicity of $P$ 
in $Y \oplus Z$ is odd. Since 
$\hat{\phi}\circ S = \hat\phi$, it follows from the definition of Brauer characters that also 
$\hat{\psi} \circ S = \hat \psi$ on $P$, and so $P \cong P^*$ as $H_{\k}$-modules. 

As in Section \ref{witt}, let $Y_1$ be an $H_{\k}$-submodule of $Y$ which is maximal subject to 
$\langle Y_1, Y_1 \rangle_Y = 0$. Then the multiplicity of $P$ in $Y_1$ equals the multiplicity 
of $P$ in $Y /{Y_1}^{\perp}$ (since $Y_1^* \cong Y /{Y_1}^{\perp}$ and $P^* =P$), and so the parity 
of the multiplicity of $P$ in $Y$ equals the parity of the multiplicity of $P$ in the Witt kernel 
$Y' = {Y_1}^{\perp}/ Y_1$. 

The same argument applies to $Z$, and thus the multiplicity of $P$ in 
$Y' \oplus Z'$ is odd.

For the second part, by Section \ref{witt} we know that $Y'$ is completely reducible, and thus if $P$ appears in $Y'$, the non-degenerate 
bilinear form on $Y'$  restricts to a non-degenerate form on $P$. By uniqueness, this form must 
agree with the given form on $P$, and thus $P$ and $Y'$, and so $P$ and $Y$, have the same type. 

Similarly, if $P$ appears in $Z'$, then $P$ and $V$ have the same type. But since $P$ appears an odd number of times in $Y' \oplus Z'$, it must appear in either $Y'$ or $Z'$. 
\end{proof}

\begin{thm} \label{main} Let $\hat{P}$ be a self-dual simple $H_{\k}$-module, and let $\hat\phi$ be its Brauer character. 
Then there is an irreducible $\K$-character $\hat{\chi}$ of $H_{\K}$ such that 

(1) $\hat{\chi}^* = \hat\chi,$ and

(2) $d (\hat{\chi} |_{\CB_{p'}}, \hat \phi)$ is odd.

\noindent Moreover if $\hat\chi$ is any irreducible $\K$-character of $H_{\K}$ satisfying (1) and (2), 
then $\nu(\hat\chi) = \nu(\hat{P}).$
\end{thm}

To prove the theorem, it will suffice to show that $\hat\chi$ exists, since the equality 
$\nu(\hat\chi) = \nu(\hat{P})$ follows from Proposition \ref{old3.3}.

We follow the outline of Thompson's argument, although we must look at the $RF_x$-blocks separately. We know that for some $x = x_q$, $\hat{P}$ is 
induced from a simple $\k F_x$-module $P$.  Let $B_x$ be the block of $RF_x$ containing the 
Brauer character $\phi$ of $P$, let $\{ \chi_1, \ldots ,\chi_m \}$ be all of the irreducible 
$\K$-characters in $B_x$, and let $\{ \phi_1, \ldots ,\phi_n \}$ be all of the irreducible Brauer 
characters in $B_x$. 

Let $D_x$ be the decomposition matrix of the $\chi_i$ with respect to the $\phi_j$, and 
$C_x = D_x^t D_x$ the Cartan matrix. From Brauer's theorem \ref{brauer}, $Det(C_x)$ is a power of $p$ and so is odd since $p$ is odd. 

Lifting this set-up to $H_{\K}$, $\hat{B}_x$ is the block of $R\CB$ containing the Brauer character 
$\hat\phi$ of $\hat{P}$, $\{ \hat\chi_1, \ldots ,\hat\chi_m \}$ are all the irreducible $\K$ characters in 
$\hat{B}_x$, and $\{ \hat\phi_1, \ldots ,\hat\phi_n \}$ are all of the irreducible Brauer characters in 
$\hat{B}_x$. 

Choose notation so that $n = 2n_1 + n_2$, where 
$\{\hat\phi_1, \hat\phi_2 \},  \{\hat\phi_3, \hat\phi_4 \}, \ldots, \{\hat\phi_{2n_1 -1}, \hat\phi_{2n_1}\}$, 
are pairs of non self-dual characters, that is, $(\hat\phi_{2i -1})^* =  \hat\phi_{2i}$, and $\hat\phi_{2n_1+1}, \ldots, \hat\phi_n$ are self-dual. By hypothesis, $n_2 \neq 0$ since $\hat{\phi}$ is 
one of the $\hat{\phi}_i$. 

Write $C_x$ in block form as 
$C_x = \left [\begin{array}{ll}
C_0 & C_2\\
C_2^t & C_1  \end{array}\right]$, where $C_0$ is $2n_1 \times 2 n_1$ and $C_1$ is 
$n_2 \times n_2.$

 The Theorem will now follow from the next two lemmas:

\begin{lem} $Det (C_1)$ is odd.
\end{lem}

\begin{proof} For $i = 1, 2, \ldots, n$, let $P_i$ be the projective indecomposable $\k F_x$-module whose socle has Brauer character $\phi_i$, and let $\Phi_i$ be the Brauer character of $P_i$. Then 
$c_{ij} = (\Phi_i, \Phi_j).$

Let $\s = (1, 2)(3, 4) \cdots (2n_1-1, 2n_1)\in \CS_n$; also let $\s$ denote the corresponding permutation matrix. Let $\tilde{\CS}_n$ be the set of all permutations in $\CS_n$ which do not fix 
$\{ 2n_1+1, \ldots, n\}$. Since $\tilde{\CS}_n$ is the complement in $\CS_n$ of the centralizer of $\s$, 
it follows that $\s^{-1} \tilde{\CS}_n \s = \tilde{\CS}_n, $ and $\s$ has no fixed points on $\tilde{\CS}_n$.

Looking at the matrix $C_x$, it follows that $\s^{-1} C_x \s = C_x$ since  $(\hat\phi_i)^* =  \hat\phi_{i+1}$ 
for $i = 1, 3, \ldots, 2n_1 -1$ and $(\hat\phi_i)^* =  \hat\phi_i$ for $i = 2n_1 +1, \ldots, n$. Then 

$$ Det(C_x) = Det(C_0) Det(C_1) + \sum_{\tau \in \tilde{\CS}_n} sgn(\tau) c_{1\tau(1)}c_{2\tau(2)} 
\cdots c_{n\tau(n)}.$$

Moreover $c_{ij} = c_{\s(i) \s(j)}$, again since $\s^{-1} C_x \s = C_x$.
Choose $\tau \in \tilde{\CS}_n$ and set $\tau' = \s \tau \s$. Then $\tau' \neq \tau$, and it follows that 
$$\prod_{i = 1}^n c_{i\tau(i)} = \prod_{i = 1}^n c_{i\tau'(i)}.$$

Thus $Det (C_x) \equiv Det (C_0) Det (C_1)$ (mod 2). This proves the Lemma.
\end{proof} 

 \begin{lem}
 For each $j = 2n_1 +1, \ldots, n,$ there exists $i \in \{1, 2, \ldots, m\}$ such that $\hat{\chi}_i^* = 
 \hat{\chi}_i$ and 
the decompostion number $d_{ij}= d (\hat{\chi_i} |_{\CB_{p'}}, \hat \phi_j)$ is odd.
 \end{lem} 
 
 \begin{proof} Let $m = 2m_1 + m_2$, where the notation is chosen so that $\{\hat\chi_1, \hat\chi_2 \},$ 
 $ \{\hat\chi_3, \hat\chi_4 \},$ $ \ldots, \{\hat\chi_{2m_1 -1}, \hat\chi_{2m_1}\}$, 
are pairs of non self-dual characters, that is, $(\hat\chi_{2i -1})^* =  \hat\chi_{2i}$, and $\hat\chi_{2m_1+1}, \ldots, \hat\chi_m$ are self-dual.

Suppose  $d_{ij} \equiv 0$ (mod 2), for all $i = 2m_1 +1, \ldots, m.$ Then for each 
$k \in \{1, 2, \ldots, n\}$, we have 
$$c_{jk} = \sum_{i = 1}^m d_{ij} d_{ik} \equiv \sum_{i = 1}^{2m_1} d_{ij} d_{ik}.$$
On the other hand, $\phi_j^* = \phi_j$ and if $k \in \{ 2n_1 +1, \ldots, n\}$ then $\phi_k^* = \phi_k$ and so
$$ d_{ij} = d_{i+1, k}, \ d_{ik} = d_{i+1, k}, \ i = 1, 3, \ldots, 2m_1-1,$$
hence $c_{jk} \equiv 0$ (mod 2), for all such $k$. This means that some row of $C_1$ 
consists of even entries. This violates the previous lemma. 
 
 \end{proof}

 As in \cite[Theorem 4.4]{GM}, we have the following consequence: 

\begin{cor} \label{orth} Consider the bismash products as above.

\noindent (1) If all irreducible $H_{\C}$-modules have indicator $+1$, the same
is true for all irreducible $H_{\k}$-modules.

\noindent (2) If all irreducible $H_{\C}$-modules have indicator $0$ or 1,
the same
is true for all irreducible $H_{\k}$-modules.

\noindent (3) If all irreducible $H_{\C}$-modules are self dual, the same is true
for all irreducible $H_{\k}$-modules.
\end{cor}

\begin{proof} (3) This follows by Lemma \ref{Hlincom}.
 
Now consider  (2).  By Theorem \ref{main}
there are no irreducible $kG$-modules $V$ with $\nu(V)=-1$.
So (2) follows immediately.

Now (1) follows by (2) and (3).

\end{proof}


\section{Applications to the symmetric group}

In this section we apply the results of Section 6 to bismash products constructed from some specific groups.

Let $\CS_n$ be the symmetric group of degree n, consider 
$\CS_{n-1} \subset \CS_n$ by letting any $\sigma\in \CS_{n-1}$ fix $n$, and let $C_n = \langle z \rangle$, 
the cyclic subgroup of $\CS_n$ generated by the $n$-cycle 
$z =  (1,2, \ldots,n)$. Then $\CS_n = \CS_{n-1}C_n = C_n\CS_{n-1}$ shows that $Q = \CS_n$ is factorizable. Thus we may construct the bismash product $H_{n, \E} := \E^{C_n}\# \E S_{n-1}$.  It was shown 
in \cite{JM} that if $\E$ is algebraically closed of characteristic 0, then $H_n$ is totally orthogonal; 
that is, every irreducible module has indicator +1. 

\begin{cor} Let $\k$ be algebraically closed of characteristic $p>0$ and let 
$H_{n,\k} := \k^{C_n}\# \k S_{n-1}$. Then $H_{n,\k}$ is totally orthogonal. 
\end{cor}

\begin{proof} Apply Corollary \ref{orth} to the characteristic 0 result of \cite{JM} mentioned above. 
\end{proof}

\begin{remark}  {\rm In \cite{GM} it is proved that $D(G)$ is totally orthogonal for any finite real reflection group $G$ over any algebraically closed field. Corollary \ref{orth} shows that this result in characteristic $p>0$  follows from the case of characteristic 0, which is somewhat easier to prove. When $G = S_n$, the 
characteristic 0 case was shown in \cite{KMM}.}
\end{remark}

We close with a question. 

\begin{ques} {\rm It would be interesting to know if our results could be extended to bicrossed products. However to 
extend our proof one would need a theory of Brauer characters for twisted group algebras (that is, for 
projective representations) which includes a version of Brauer's theorem on the Cartan matrix. }
\end{ques}

\vs

{\bf Acknowledgement:} The authors would like to thank I. M. Isaacs, S. Witherspoon, and especially 
R. M. Guralnick, for helpful comments. 



\begin{thebibliography}{FGSV}


\bibitem[CR]{CR} C. Curtis and I. Reiner, {\it Methods of Representation Theory} Vol I, 
Wiley-Interscience, New York, 1981. 

\bibitem[DPR]{DPR} R. Dijkgraaf, V. Pasquier, P. Roche, QuasiHopf algebras,
group cohomology and orbifold models, {\it Nucl. Phys. B Proc. Suppl} 18B (1990),
60 - 72.

\bibitem[EG]{EG} P. Etingof and S. Gelaki, On finite-dimensional semisimple and 
cosemisimple Hopf algebras in positive characteristic,  
{\it Internat. Math. Res. Notices} 16 (1998), 851Ð864.


\bibitem[GM]{GM} R. Guralnick and S. Montgomery, Frobenius-Schur Indicators for subgroups 
and the Drinfel'd double of Weyl groups, {\it AMS Transactions} 361 (2009), 3611--3632. 

\bibitem[I]{I} I. M. Isaacs, {\it Character Theory of Finite Groups}, Dover, New York, 1976. 

\bibitem[JL]{JL}  G. James and M. Liebeck, {\it Representations and characters
of groups}, Cambridge Mathematical Textbooks, Cambridge University
Press, Cambridge, 1993. 

\bibitem[JM]{JM} A. Jedwab and S. Montgomery, Representations of some Hopf algebras associated to the symmetric group $S_n$, {\it Algebras and Representation Theory} 12 (2009), 1-17. 

\bibitem[Ka]{Ka} G. I. Kac, Group extensions which are ring groups (Russian), {\it Mat. Sbornik} N. S. 
76 (118), 473-496 (1978). English translation in {\it Mathematics of the USSR-Sbornik}, 451-474 (1958)


\bibitem[KMM]{KMM} Y. Kashina, G. Mason and S. Montgomery, Computing the Frobenius-Schur
indicator for abelian extensions of Hopf algebras, {\it J. Algebra}
251 (2002), 888--913.

\bibitem[KSZ1]{KSZ1} Y. Kashina, Y. Sommerh\"auser, and Y. Zhu, Self-dual modules
of semisimple Hopf algebras,  {\it J. Algebra} 257 (2002), 88--96.



\bibitem[LM]{LM}  V. Linchenko and S. Montgomery, A Frobenius-Schur theorem
for Hopf algebras, {\it Algebras and Representation Theory}, 3
(2000), 347--355.


\bibitem[M]{M}  G. Mason, The quantum double of a finite group and its
role in conformal field theory, Groups '93 Galway/St. Andrews,
Vol. 2, 405--417, {\it London Math. Soc. Lecture Note Ser.}, 212,
Cambridge Univ. Press, Cambridge, 1995.


\bibitem[Ma1]{Ma1}  A. Masuoka, Calculations of some groups of Hopf algebra extensions,
{\it J. Algebra} 191 (1997), 568--588.

\bibitem[Ma2]{Ma2}  A. Masuoka, Extensions of Hopf algebras,
{\it Trabajos de Matem\'{a}tica} 41/99, FaMAF, Universdad Nacional
de C\'{o}rdoba, Argentina, 1999.

\bibitem[Ma3]{Ma3}  A. Masuoka, Hopf algebra extensions and cohomology,
{\it New directions in Hopf algebras}, Mathematical Sciences
Research Institutes Publications, Vol. 3, Cambridge University
Press, Cambridge, 2002, pp. 167--209.

\bibitem[Mo]{Mo}  S. Montgomery, {\it Hopf Algebras and their Actions on Rings},
CbMS Lectures, Vol. 82, AMS, Providence, RI, 1993.

\bibitem[MoW]{MoW}  S. Montgomery and S. Witherspoon,  Irreducible representations of crossed 
products, {\it J. Pure Appl. Algebra  129}  (1998), 315--326.

\bibitem[N1]{N1}  S. Natale, On group-theoretical Hopf algebras and exact factorizations of 
finite groups,  {\it J. Algebra} 270 (2003), 199-211.

\bibitem[N2]{N2}  S. Natale, Frobenius-Schur indicators for a class of fusion categories, 
{\it Pacific J. Math} 221 (2005), 363--377.

\bibitem[Nv]{Nv} G. Navarro, {\it Characters and Blocks of Finite Groups}, LMS Lecture Note 
Series 250, Cambridge University Press, Cambridge, 1998. 

\bibitem[S]{S} T. A. Springer, A construction of representations of Weyl groups, {\it Inventiones 
Math} 44 (1978), 279-293. 


\bibitem[Ta]{Ta}  M. Takeuchi, Matched pairs of groups and bismash products of Hopf algebras, 
{\it Comm. in Algebra} 9 (1981), 84--882.

\bibitem[Th]{Th}  J. G. Thompson, Some finite groups which appear as Gal$L/K$, where 
$K\subseteq \Q (\mu_n)$. Group theory, beijing 1984, 210Ð230, 
{\it Lecture Notes in Math.}, 1185, Springer, Berlin, 1986. 

\bibitem[W1]{W2} S. Witherspoon,  The representation ring of the quantum double of a finite group,
{\it J. Algebra  179}  (1996), 305--329.

\bibitem[W2]{W2} S. Witherspoon,  Products in Hochschild cohomology and Grothendieck rings of 
group corssed products,
{\it Advances in Math 185}  (2004), 136--158.


\end{thebibliography}
\end{document}